\documentclass[11pt, letterpaper,reqno]{amsart}

\usepackage{amsmath}
\usepackage{amssymb}
\usepackage{amsfonts}
\usepackage{ dsfont }
\usepackage{cases}
\usepackage{cite}
\usepackage[hyphens]{url}
\usepackage{hyperref}
\usepackage{amsthm}
\usepackage{upgreek}
\usepackage{color}
\usepackage{xcolor}
\definecolor{color0}{gray}{.50}
\definecolor{color1}{rgb}{0,.2,.8}
\definecolor{color2}{rgb}{1,.2,0}
\definecolor{color3}{rgb}{.8,.5,1}

\numberwithin{equation}{section}

\newtheorem{theorem}{Theorem}[section]

\newtheorem{lemma}[theorem]{Lemma}
\newtheorem{proposition}[theorem]{Proposition}
\newtheorem{definition}[theorem]{Definition}
\newtheorem{assumption}[theorem]{Assumption}
\newtheorem{remark}[theorem]{Remark}

\newcommand{\leref}{Lemma~\ref}

\newcommand{\prref}{Proposition~\ref}
\newcommand{\thref}{Theorem~\ref}
\newcommand{\reref}{Remark~\ref}
\newcommand{\asref}{Assumption~\ref}

\newcommand{\E}{\mathbb{E}}
\newcommand{\PP}{\mathbb{P}}

\newcommand{\R}{\mathbb{R}}
\newcommand{\eps}{\epsilon}

\newcommand{\cT}{\mathcal{T}}
\newcommand{\T}{\mathbb{T}}
\newcommand{\kT}{\mathfrak{T}}
\newcommand{\cF}{\mathcal{F}}
\newcommand{\N}{\mathbb{N}}

\title[]{Non-zero-sum stopping games in continuous time}
\author[]{Zhou Zhou}\thanks{The author would like to thank Erhan Bayraktar for helpful discussions}
\address{Institute for Mathematics and its Applications, University of Minnesota}
\email{zhouzhou@umich.edu}
\date{\today}

\begin{document}
\maketitle
\begin{abstract}
On a filtered probability space $(\Omega,\mathcal{F},(\mathcal{F}_t)_{t\in[0,\infty]},\PP)$, we consider the two-player non-zero-sum stopping game $u^i:=\E[U^i(\rho,\tau)],\ i=1,2$, where the first player choose a stopping strategy $\rho$ to maximize $u^1$ and the second player chose a stopping strategy $\tau$ to maximize $u^2$. Unlike the Dynkin game, here we assume that $U(s,t)$ is $\mathcal{F}_{s\vee t}$-measurable. Assuming the continuity of $U^i$ in $(s,t)$, we show that there exists an $\eps$-Nash equilibrium for any $\eps>0$.
\end{abstract}

\section{Introduction}
On a filtered probability space $(\Omega,\mathcal{F},(\mathcal{F}_t)_{t\in[0,\infty]},\PP)$, we consider the two-player non-zero-sum stopping game $u^i(\rho,\tau):=\E[U^i(\rho,\tau)],\ i=1,2$, where the first player choose a stopping strategy $\rho$ to maximize $u^1$ and the second player chose a stopping strategy $\tau$ to maximize $u^2$. Unlike the Dynkin game, here we assume that $U(s,t)$ is $\mathcal{F}_{s\vee t}$-measurable instead of $\mathcal{F}_{s\wedge t}$-measurable. That is, the game ends at the maximum of $\rho$ and $\tau$. Moreover, $\rho$ and $\tau$ are not stopping times. They are some strategies satisfying some non-anticipativity condition.  By assuming the continuity of $U^i$ in $(s,t)$, we show that there exists an $\eps$-Nash equilibrium for any $\eps>0$.

We first provide some results in the zero-sum case. In particular, we construct an $\eps$-saddle point. (A more general version for the zero-sum case has been studied in \cite{ZZ7}.) Then we construct an $\eps$-Nash equilibrium of the non-zero-sum game by using the $\eps$-saddle points in the zero-sum case. 

To avoid the technical difficulties stemming from the verification of the path regularity of some related processes, we assume that the state space is at most countable. Besides, we assume that the payoff function $U^i$ is uniformly continuous in $(t,s)$. These assumptions, though restricted, help us concentrate on the interesting parts of the paper. It can be expected that the result still holds in a more general framework (The right continuity of $U^i$ is always in force).

Compared to the non-zero-sum Dynkin game (see e.g., \cite{Zhang3,Solan}), our game is more applicable. In practice, even if a player has made the decision first, her payoff can still be affected by other players' decisions later on. Therefore, it is more reasonable to let the game end at the maximum rather than the minimum of the stopping strategies. As far as we know, this is the first paper studying the stopping games with this feature in a non-zero-sum setup. (Recently there are two papers\cite{ZZ6,ZZ7} on the stopping games with this feature in the zero-sum case.)

The paper is organized as follows. In the next section, we introduce the setup and the main result. In Section 3, we provide some results in the zero-sum case. In Section 4, we use the results from the zero-sum case to construct an $\eps$-Nash equilibrium for the non-zero-sum stopping games.
 
\section{The setup and the main result}
Let $(\Omega,\mathcal{F},(\mathcal{F}_t)_{t\in[0,\infty]},\PP)$ be a filtered probability space, where without loss of generality $\mathcal{F}=\cF_\infty=\cup_{t\in[0,\infty)}\mathcal{F}_t$, and the filtration $(\mathcal{F}_t)_{t\in[0,\infty]}$ satisfies the usual conditions. To avoid the technical difficulties stemming from the verification of the path regularities of some related related processes, we assume that $\Omega$ is at most countable, and that $\PP$ is supported on $\Omega$. Let $\cT$ be the set of stopping times taking values in $[0,\infty]$. For any $\sigma\in\cT$, denote $\cT_\sigma$ (resp. $\cT_{\sigma+}$) as the set of stopping times that is no less (resp. strictly greater) than $\sigma$. Let
$$\T:=\{\phi:[0,\infty]\times\Omega\mapsto[0,\infty]:\ \phi\text{ is }\mathcal{B}([0,\infty))\otimes\mathcal{F}-\text{measurable, and }\forall t\in[0,\infty),\ \phi(t,\cdot)\in\cT_{t+}\}.$$

\begin{definition}
$\rho=(\rho_0,\rho_1)$ is said to be a stopping strategy, if $\rho_0\in\cT$ and $\rho_1\in\T$. Denote the set of stopping strategies as $\kT$.
\end{definition}
For $\sigma\in\cT$ and $\phi\in\kT$, denote $\phi(\sigma):=\phi(\sigma(\cdot),\cdot)$. For any $\rho=(\rho_0,\rho_1),\tau=(\tau_0,\tau_1)\in\kT$, denote
$$\rho[\tau]:=\rho_0 1_{\{\rho_0\leq\tau_0\}}+\rho_1(\tau_0) 1_{\{\rho_0>\tau_0\}}.$$
For $i=1,2$, let $U^i:[0,\infty)\times[0,\infty)\times\Omega\mapsto\R$, such that $U(s,t)$ is $\mathcal{F}_{s\vee t}$-measurable. For simplicity, we assume that $U^i$ is bounded, and for any $s,t\in[0,\infty)$, the limits
$$U^i(s,\infty):=\inf_{b\rightarrow\infty}U^i(s,b),\quad U^i(\infty,t):=\inf_{a\rightarrow\infty}U^i(a,t),\quad U^i(\infty,\infty):=\inf_{a,b\rightarrow\infty}U^i(a,b)$$
exist, $i=1,2$. We make the following assumption on $U^i$.

\begin{assumption}\label{a1}
There exists a bounded non-decreasing function $r:[0,\infty]\mapsto\R_+$ with $\lim_{h\searrow 0}r(h)=r(0)=0$, and
$$\left|U^i(s,t,\cdot)-U^i(s',t',\cdot)\right|\leq r\left(|s-s'|+|t-t'|\right),\quad\forall s,t,s',t'\in[0,\infty],\quad i=1,2.$$
\end{assumption}

Consider the two-player non-zero-sum stopping game
\begin{eqnarray}
\label{e1} u^i(\rho,\tau)&=&\E\left[U^i(\rho[\tau],\tau[\rho])\right]\\
\notag&=&\E\left[U^i(\rho_0,\tau_1(\rho_0))1_{\{\rho_0<\tau_0\}}+U^i(\rho_1(\tau_0),\tau_0)1_{\{\rho_0>\tau_0\}}+U^i(\rho_0,\rho_0)1_{\{\rho_0=\tau_0\}}\right],
\end{eqnarray}
$\rho=(\rho_0,\rho_1),\tau=(\tau_0,\tau_1)\in\kT,\ i=1,2$. Here the first player choose $\rho$ to maximize $u^1$ and the second player choose $\tau$ to maximize $u^2$.

Recall the definition of an $\eps$-Nash equilibrium.
\begin{definition}\label{d1}
For $\eps>0$, $(\rho^*,\tau^*)\in\kT^2$ is said to be an $\eps$-Nash equilibrium for the stopping game \eqref{e1}, if for any $\rho,\tau\in\kT$,
$$u^1(\rho,\tau^*)\leq u^1(\rho^*,\tau^*)+\eps\quad\text{and}\quad u^2(\rho^*,\tau)\leq u^2(\rho^*,\tau^*)+\eps.$$
\end{definition}

Below is the main result of this paper.
\begin{theorem}\label{t1}
Under \asref{a1}, there exists an $\eps$-Nash equilibrium for \eqref{e1} for any $\eps>0$.
\end{theorem}

\section{Zero-sum case}
In this section, we consider \eqref{e1} in the zero-sum case,  i.e., we assume $U^1=-U^2:=U$ throughout this section. For $t\in[0,\infty]$, let
$$X_t:=\inf_{\sigma\in\cT_t}\E_t[U(t,\sigma)],\quad Y_t:=\sup_{\sigma\in\cT_t}\E_t[U(\sigma,t)]\quad\text{and}\quad Z_t:=U(t,t),$$
where $\E_\sigma[\cdot]:=\E[\cdot|\cF_\sigma]$ for any $\sigma\in\cT$. Obviously,
\begin{equation}\label{e11}
X_t\leq Z_t\leq Y_t.
\end{equation}
\begin{lemma}\label{l1}
The processes $(X_t)_{t\in[0,\infty]}$ and $(Y_t)_{t\in[0,\infty]}$ are right continuous. 
\end{lemma}
\begin{proof}
Fix $t\in[0,\infty)$. Let $t_n\searrow t$ with $|t_n-t|<1/n$. Let
$$\tilde X_v:=\inf_{\sigma\in\cT_v}\E_v\left[U(t,\sigma)\right],\quad v\in\cT_t.$$
Then $\tilde X_t=X_t$. We have that
$$\limsup_{n\rightarrow\infty}|X_{t_n}-X_{t}|\leq\limsup_{n\rightarrow\infty}|X_{t_n}-\tilde X_{t_n}|+\limsup_{n\rightarrow\infty}|\tilde X_{t_n}-\tilde X_t|.$$
Now
\begin{eqnarray}
\notag\limsup_{n\rightarrow\infty}|X_{t_n}-\tilde X_{t_n}|&=&\limsup_{n\rightarrow\infty}\left|\inf_{\sigma\in\cT_{t_n}}\E_{t_n}\left[U(t_n,\sigma)\right]-\inf_{\sigma\in\cT_{t_n}}\E_{t_n}\left[U(t,\sigma)\right]\right|\\
\notag&\leq&\limsup_{n\rightarrow\infty}\sup_{\sigma\in\cT_{t_n}}\E_{t_n}\left|U(t_n,\sigma)-U(t,\sigma)\right|\\
\notag&\leq&\limsup_{n\rightarrow\infty}\,r(1/n)\\
\notag&=&0.
\end{eqnarray}
By \cite[Theorem D.7]{KS2}, $\limsup_{n\rightarrow\infty}|\tilde X_{t_n}-\tilde X_t|=0$. Hence, the result follows.
\end{proof}

Let
$$V(\rho_0,\tau_0):=\E\left[X_{\rho_0} 1_{\{\rho_0<\tau_0\}}+Y_{\tau_0} 1_{\{\rho_0>\tau_0\}}+Z_{\rho_0} 1_{\{\rho_0=\tau_0\}}\right],\quad\rho,\tau\in\cT,$$
and consider the zero-sum Dynkin game
\begin{equation}\label{e12}
\underline v:=\sup_{\rho_0\in\cT}\inf_{\tau_0\in\cT}V(\rho_0,\tau_0)\quad\text{and}\quad\overline v:=\inf_{\tau_0\in\cT}\sup_{\rho_0\in\cT}V(\rho_0,\tau_0).
\end{equation}
By \eqref{e11} and \leref{l1},  and the game has a value and there exists an $\eps$-saddle point $(\rho_0^\eps,\tau_0^\eps)\in\cT^2$ for any $\eps>0$,. That is, $\underline v=\overline v$ and for any $\rho_0,\tau_0\in\cT$,
$$V(\rho_0,\tau_0^\eps)-\eps<V(\rho_0^\eps,\tau_0^\eps)<V(\rho_0^\eps,\tau_0)+\eps.$$

Now let $h>0$. For each $n\in\N$, let $\bar\rho(nh)\in\cT_{nh}$ be the $\eps$-optimizer for $Y_{nh}$. That is,
$$\E_{nh}\left[U(\bar\rho(nh),nh)\right]\geq\sup_{\sigma\in\cT_{nh}}\E\left[U(\sigma,nh)\right]-\eps.$$
Define $\rho_h:[0,\infty]\times\Omega\mapsto[0,\infty]$,
\begin{equation}\label{e14}
\rho_h(t):=\bar\rho\left(\left(\left[t/h\right]+1\right)h\right),\quad t\in[0,\infty),
\end{equation}
and $\rho_h(\infty)=\infty$. Then it is easy to see that $\rho_h\in\T$. Similarly, let $\bar\tau(nh)$ be the $\eps$-optimizer for $X_{nh}$, and define $\tau_h:[0,\infty]\times\Omega\mapsto[0,\infty]$,
\begin{equation}\label{e15}
\tau_h(t):=\bar\tau\left(\left(\left[t/h\right]+1\right)h\right),\quad t\in[0,\infty),
\end{equation}
and $\tau_h(\infty)=\infty$. Observe that $\rho_h(\cdot,\omega)$ and $\tau_h(\cdot,\omega)$ are right continuous for each $\omega\in\Omega$. This right continuity will play an important role in the proof of \thref{t1} in the next section.

\begin{lemma}\label{l2}
For any $\eps>0$, there exists $h'>0$, such that for any $h\in(0,h')$,
\begin{equation}\label{e3}
X_{\rho_0}>\E_{\rho_0}\left[U(\rho_0,\tau_h(\rho_0))\right]-2\eps\quad\text{and}\quad Y_{\tau_0}<\E_{\tau_0}\left[U(\rho_h(\tau_0),\tau_0)\right]+2\eps,\quad\forall\rho_0,\tau_0\in\cT,
\end{equation}
\end{lemma}
\begin{proof}
For $\tau_0\in\cT$, denote $\tau_{0h}=([\tau_0/h]+1)h$. Then
\begin{eqnarray}
\notag&&\left|Y_{\tau_0}-\E_{\tau_0}\left[U(\rho_h(\tau_0),\tau_0)\right]\right|\\
\notag&\leq&\left|\sup_{\sigma\in\cT_{\tau_0}}\E_{\tau_0}\left[U(\sigma,\tau_0)\right]-\E_{\tau_0}\left[U(\bar\rho(\tau_{0h}),\tau_{0h})\right]\right|+\left|\E_{\tau_0}\left[U(\bar\rho(\tau_{0h}),\tau_{0h})\right]-\E_{\tau_0}\left[U(\bar\rho(\tau_{0h}),\tau_0)\right]\right|\\
\notag &\leq&\left|\sup_{\sigma\in\cT_{\tau_0}}\E_{\tau_0}\left[U(\sigma,\tau_0)\right]-\E_{\tau_0}\left[\E_{\tau_{0h}}\left[U(\bar\rho(\tau_{0h}),\tau_{0h})\right]\right]\right|+r(h)\\
\notag&\leq&\left|\sup_{\sigma\in\cT_{\tau_0}}\E_{\tau_0}\left[U(\sigma,\tau_0)\right]-\E_{\tau_0}\left[\sup_{\sigma\in\cT_{\tau_{0h}}}\E_{\tau_{0h}}\left[U(\sigma,\tau_{0h})\right]\right]\right|+\eps+r(h)\\
\notag&\leq&\left|\sup_{\sigma\in\cT_{\tau_0}}\E_{\tau_0}\left[U(\sigma,\tau_0)\right]-\E_{\tau_0}\left[\sup_{\sigma\in\cT_{\tau_{0h}}}\E_{\tau_{0h}}\left[U(\sigma,\tau_0)\right]\right]\right|+\eps+2r(h)\\
\notag&=&\left|\sup_{\sigma\in\cT_{\tau_0}}\E_{\tau_0}\left[U(\sigma,\tau_0)\right]-\sup_{\sigma\in\cT_{\tau_{0h}}}\E_{\tau_0}\left[U(\sigma,\tau_0)\right]\right|+\eps+2r(h)\\
\notag&=&\left|\sup_{\sigma\in\cT_{\tau_0}}\E_{\tau_0}\left[U(\sigma,\tau_0)\right]-\sup_{\sigma\in\cT_{\tau_0}}\E_{\tau_0}\left[U(\sigma\vee\tau_{0h},\tau_0)\right]\right|+\eps+2r(h)\\
\notag&\leq&\sup_{\sigma\in\cT_{\tau_0}}\E_{\tau_0}\left|U(\sigma,\tau_0)-U(\sigma\vee\tau_h,\tau_0)\right|+\eps+2r(h)\\
\notag&\leq&3r(h)+\eps.
\end{eqnarray}
Now choose $h'>0$, such that $r(h')<\eps/3$. Then the result follows.
\end{proof}

\begin{proposition}\label{p1}
For any $\eps>0$, let $\rho^\eps=(\rho_0^\eps,\rho_h)$ and $\tau^\eps=(\tau_0^\eps,\tau_h)$, where $(\rho_0^\eps,\tau_0^\eps)$ is a saddle point for \eqref{e12}, and $h$ is small enough such that \eqref{e3} holds. Then $(\rho^\eps,\tau^\eps)$ is a $5\eps$-Nash equilibrium for \eqref{e1} when $U^1=-U^2=U$. Hence \thref{t1} holds in the zero-sum case.
\end{proposition}
\begin{proof}
Let $\rho=(\rho_0,\rho_1)\in\kT$. We have that
\begin{eqnarray}
\notag u^1(\rho,\tau^\eps)&=&\E\left[U(\rho_0,\tau_h(\rho_0))1_{\{\rho_0<\tau_0^\eps\}}+U(\rho_1(\tau_0^\eps),\tau_0^\eps)1_{\{\rho_0>\tau_0^\eps\}}+U(\rho_0,\rho_0)1_{\{\rho_0=\tau_0^\eps\}}\right]\\
\notag &=&\E\left[\E_{\rho_0}\left[U(\rho_0,\tau_h(\rho_0))\right]1_{\{\rho_0<\tau_0^\eps\}}+\E_{\tau_0^\eps}\left[U(\rho_1(\tau_0^\eps),\tau_0^\eps)\right]1_{\{\rho_0>\tau_0^\eps\}}+U(\rho_0,\rho_0)1_{\{\rho_0=\tau_0^\eps\}}\right]\\
\notag &\leq&\E\left[X_{\rho_0}1_{\{\rho_0>\tau_0^\eps\}}+Y_{\tau_0^\eps}1_{\{\rho_0<\tau_0^\eps\}}+Z_{\rho_0}1_{\{\rho_0=\tau_0^\eps\}}\right]+2\eps\\
\notag &\leq&\E\left[X_{\rho_0^\eps}1_{\{\rho_0^\eps>\tau_0^\eps\}}+Y_{\tau_0^\eps}1_{\{\rho_0^\eps<\tau_0^\eps\}}+Z_{\rho_0^\eps}1_{\{\rho_0^\eps=\tau_0^\eps\}}\right]+3\eps\\
\notag &\leq&\E\left[\E_{\rho_0^\eps}\left[U(\rho_0^\eps,\tau_h(\rho_0^\eps))\right]1_{\{\rho_0^\eps<\tau_0^\eps\}}+\E_{\tau_0^\eps}\left[U(\rho_h(\tau_0^\eps),\tau_0^\eps)\right]1_{\{\rho_0^\eps>\tau_0^\eps\}}+U(\rho_0^\eps,\rho_0^\eps)1_{\{\rho_0^\eps=\tau_0^\eps\}}\right]+5\eps\\
\notag &=& u^1(\rho^\eps,\tau^\eps)+5\eps.
\end{eqnarray}
Similarly, we can show that
$$u^2(\rho^\eps,\tau)\leq u^2(\rho^\eps,\tau^\eps)+5\eps.$$
\end{proof}

\begin{remark}\label{r1}
We can also consider the sub-game of  \eqref{e1}
\begin{equation}\label{e13}
u_\sigma^i(\rho,\tau)=\E_\sigma\left[U^i(\rho[\tau],\tau[\rho])\right],\quad(\rho,\tau)\in\kT_\sigma^2,\quad\sigma\in\cT,\quad i=1,2,
\end{equation}
where
$$\kT_\sigma:=\{(\rho_0,\rho_1)\in\kT:\ \rho_0\geq\sigma\}.$$
Follow the same proof as above, we can show that $((\rho_\sigma^\eps,\rho_h),(\tau_\sigma^\eps,\rho_h))\in\kT_\sigma^2$ is a $5\eps$-Nash equilibrium for \eqref{e13} when $U^1=-U^2=U$, where $(\rho_\sigma^\eps,\tau_\sigma^\eps)$ is an $\eps$-saddle point for the sub-Dynkin game
\begin{eqnarray}
\notag v_\sigma&:=&\sup_{\rho\in\cT_\sigma}\inf_{\tau\in\cT_\sigma}\E_\sigma\left[X_{\rho_0} 1_{\{\rho_0<\tau_0\}}+Y_{\tau_0} 1_{\{\rho_0>\tau_0\}}+Z_{\rho_0} 1_{\{\rho_0=\tau_0\}}\right]\\
\notag &=& \inf_{\tau\in\cT_\sigma}\sup_{\rho\in\cT_\sigma}\E_\sigma\left[X_{\rho_0} 1_{\{\rho_0<\tau_0\}}+Y_{\tau_0} 1_{\{\rho_0>\tau_0\}}+Z_{\rho_0} 1_{\{\rho_0=\tau_0\}}\right].
\end{eqnarray}
Moreover, as can be seen from the proof of \prref{p1}, for any $(\rho_0,\rho_1)\in\kT_\sigma$,
$$u_\sigma^1((\rho_0,\rho_1),(\tau_\sigma^\eps,\tau_h))\leq v_\sigma+4\eps.$$
\end{remark}

\section{Proof of \thref{t1}}
Throughout this section, we fix $\eps>0$. For $t\in[0,\infty]$, let
\begin{equation}\notag
X_t^1:=\inf_{\sigma\in\cT_t}\E_t\left[U^1(t,\sigma)\right],\quad Y_t^1:=\sup_{\sigma\in\cT_t}\E_t\left[U^1(\sigma,t)\right],\quad Z_t^1=U^1(t,t),
\end{equation}
and
\begin{equation}\notag
X_t^2:=\sup_{\sigma\in\cT_t}\E_t\left[U^2(t,\sigma)\right],\quad Y_t^2:=\inf_{\sigma\in\cT_t}\E_t\left[U^2(\sigma,t)\right],\quad Z_t^2=U^2(t,t).
\end{equation}
Similar to \eqref{e14} and \eqref{e15}, for $i=1,2$, let $\bar\rho^i(nh)$ and $\bar\tau^i(nh)$ be $\eps$-optimizers for $Y_{nh}^i$ and $X_{nh}^i$ respectively, and define
$$\rho_h^i(t):=\bar\rho^i(([t/h]+1)h)\quad\text{and}\quad\tau_h^i(t):=\bar\tau^i(([t/h]+1)h).$$
By \leref{l2}, we choose $h$ small enough with $r(h)<\eps/3$, such that for $i=1,2$, and any $\rho_0,\tau_0\in\cT$,
$$\left|X_{\rho_0}^i-\E_{\rho_0}\left[U^i(\rho_0,\tau_h^i(\rho_0))\right]\right|\leq 2\eps\quad\text{and}\quad\left|Y_{\tau_0}^i-\E_{\tau_0}\left[U^i(\rho_h^i(\tau_0),\tau_0)\right]\right|\leq 2\eps.$$

For $t\in[0,\infty]$, let us define
\begin{equation}\notag
W_t^1:=\E_t\left[U^1(t,\tau_h^2(t))\right]\quad\text{and}\quad W_t^2:=\E_t\left[U^2(\rho_h^1(t),t)\right].
\end{equation}

\begin{lemma}\label{l3}
For $i=1,2$, the process $W^i=(W_t^i)_{t\in[0,\infty]}$ is right continuous.
\end{lemma}
\begin{proof}
Fix $t\in[0,\infty)$. Let $t_n\searrow t$ with $|t_n-t|<(1/n)\wedge([t/h]+1)h-t)$. Obviously, $\tau_h^2(t_n)=\tau_h^2(t)$. Then
$$|W_{t_n}^1-W_t^1|\leq\left|\E_{t_n}\left[U(t_n,\tau_h^2(t))\right]-\E_{t_n}\left[U(t,\tau_h^2(t))\right]\right|+\left|\E_{t_n}\left[U(t,\tau_h^2(t))\right]-\E_t\left[U(t,\tau_h^2(t))\right]\right|.$$
On the right-hand-side of the inequality above, denote the first term as $I$ and the second term as $J$. We have that
$$I\leq\E_{t_n}\left|U(t_n,\tau_h^2(t))-U(t,\tau_h^2(t))\right|\leq r(1/n)\rightarrow 0.$$
Now consider $J$. Observe that the process
$$\left(f_s:=\E_{t+s}\left[U(t,\tau_h^2(t))\right]\right)_{s\in[0,\infty]}$$
is a martingale w.r.t. the filtration $(\cF_{t+s})_{s\in[0,\infty]}$ satisfying the usual conditions. By \cite[Theorem 3.13, page 16]{KS}, $(f_s)_{s\in[0,\infty]}$ is right continuous. Hence, $J\rightarrow 0$.
\end{proof}

Consider the sub-Dynkin game $G_t^1$
\begin{eqnarray}
\label{e16} v_t^1&:=&\sup_{\rho_0\in\cT_t}\inf_{\tau_0\in\cT_t}\E_t\left[X_{\rho_0}^1 1_{\{\rho_0<\tau_0\}}+Y_{\tau_0}^1 1_{\{\rho_0>\tau_0\}}+Z_{\rho_0}^1 1_{\{\rho_0=\tau_0\}}\right]\\
\notag &=&\inf_{\tau_0\in\cT_t}\sup_{\rho_0\in\cT_t}\E_t\left[X_{\rho_0}^1 1_{\{\rho_0<\tau_0\}}+Y_{\tau_0}^1 1_{\{\rho_0>\tau_0\}}+Z_{\rho_0}^1 1_{\{\rho_0=\tau_0\}}\right],
\end{eqnarray}
and the game $G_t^2$
\begin{eqnarray}
\label{e17} v_t^2&:=&\inf_{\rho_0\in\cT_t}\sup_{\tau_0\in\cT_t}\E_t\left[X_{\rho_0}^2 1_{\{\rho_0<\tau_0\}}+Y_{\tau_0}^2 1_{\{\rho_0>\tau_0\}}+Z_{\rho_0}^2 1_{\{\rho_0=\tau_0\}}\right]\\
\notag &=&\sup_{\tau_0\in\cT_t}\inf_{\rho_0\in\cT_t}\E_t\left[X_{\rho_0}^2 1_{\{\rho_0<\tau_0\}}+Y_{\tau_0}^2 1_{\{\rho_0>\tau_0\}}+Z_{\rho_0}^2 1_{\{\rho_0=\tau_0\}}\right].
\end{eqnarray}
By \cite[Lemma 8]{Solan}, the process $(v_t^i)_{t\in[0,\infty]}$ is right continuous for $i=1,2$. Define stopping times
$$\mu_1:=\inf\{t\geq 0:\ v_t^1\leq W_t^1+\eps\}\quad\text{and}\quad\mu_2:=\inf\{t\geq 0:\ v_t^2\leq W_t^2+\eps\}.$$
Since $X^1\leq W^1$, we have that
$$\mu_1\leq\inf\{t\geq 0:\ v_t^1\leq X_t^1+\eps\}.$$
Therefore, the process $(v_t^1)_{t\in[0,\mu_1]}$ is a sub-martingale. Similarly, $(v_t^2)_{t\in[0,\mu_2]}$ is a sub-martingale.

Let $\delta>0$ that will be chosen later on. For $i=1,2$, let $(\rho_{\mu_i+\delta}^i,\tau_{\mu_i+\delta}^i)\in\cT^2_{\mu_i+\delta}$ be an $\eps$-saddle point for game $G_{\mu_i+\delta}^i$ defined in \eqref{e16} and \eqref{e17}. Define
\[ 
\rho_0^*:=
\begin{cases} 
      \mu_1, & \text{ if }\mu_1\leq\mu_2,\\
      \rho_{\mu_2+\delta}^2, & \text{ if }\mu_1>\mu_2,\\
\end{cases}
\quad\quad
\rho_1^*(t):=
\begin{cases} 
      \rho_h^2(t), & \text{ if }t\geq\mu_1\wedge\mu_2+\delta\text{ and } \mu_1>\mu_2\\
      \rho_h^1(t), & \text{ otherwise},\\
\end{cases}
\]
and
\[ 
\tau_0^*:=
\begin{cases} 
      \tau_{\mu_1+\delta}^1, & \text{ if }\mu_1\leq\mu_2,\\
      \mu_2, & \text{ if }\mu_1>\mu_2,\\
\end{cases}
\quad\quad
\tau_1^*(t):=
\begin{cases} 
      \tau_h^1(t), & \text{ if }t\geq\mu_1\wedge\mu_2+\delta\text{ and } \mu_1\leq\mu_2\\
      \tau_h^2(t), & \text{ otherwise}.\\
\end{cases}
\]
Let $\rho^*:=(\rho_0^*,\rho_1^*)$ and $\tau^*:=(\tau_0^*,\tau_1^*)$. It can been shown that $\rho^*,\tau^*\in\kT$.

\begin{proposition}
For $\delta$ small enough, $(\rho^*,\tau^*)$ is an $18\eps$-Nash equilibrium for the non-zero-sum game \eqref{e1}. Therefore, \thref{t1} holds. 
\end{proposition}
\begin{proof}
\begin{eqnarray}
\notag u^1(\rho^*,\tau^*)&=&\E\left[U^1\left(\rho^*[\tau^*],\tau^*[\rho^*]\right)\right]\\
\notag &=& \E\left[U^1\left(\rho^*[\tau^*],\tau^*[\tau^*]\right)\left(1_{\{\mu_1\leq\mu_2\}}+1_{\{\mu_1>\mu_2\}}\right)\right]\\
\notag &=& \E\left[U^1\left(\mu_1,\tau_h^2(\mu_1)\right)1_{\{\mu_1\leq\mu_2\}}+U^1\left(\rho_h^1(\mu_2),\mu_2\right)1_{\{\mu_1>\mu_2\}}\right]\\
\notag &=& \E\left[\E_{\mu_1}\left[U^1\left(\mu_1,\tau_h^2(\mu_1)\right)\right]1_{\{\mu_1\leq\mu_2\}}+\E_{\mu_2}\left[U^1\left(\rho_h^1(\mu_2),\mu_2\right)\right]1_{\{\mu_1>\mu_2\}}\right]\\
\notag &\geq& \E\left[W_{\mu_1}^1 1_{\{\mu_1\leq\mu_2\}}+Y_{\mu_2}^1 1_{\{\mu_1>\mu_2\}}\right]-\eps.
\end{eqnarray}
Now for any $\rho=(\rho_0,\rho_1)\in\kT$, we consider $u^1(\rho,\tau^*)=\E[U^1(\rho[\tau^*],\tau^*[\rho])]$. We discuss five cases.

\textbf{Case 1}: $A_1:=\{\rho_0<\mu_1\wedge\mu_2\}$. We have that
\begin{eqnarray}
\notag \E\left[U^1\left(\rho[\tau^*],\tau^*[\rho]\right)1_{A_1}\right]&=&\E\left[U^1\left(\rho_0,\tau_h^2(\rho_0)\right)1_{A_1}\right]\\
\notag &=&\E\left[W_{\rho_0}^1 1_{A_1}\right]\\
\notag &\leq&\E\left[v_{\rho_0}^1 1_{A_1}\right]\\
\notag &=&\E\left[v_{\rho_0\wedge\mu_1\wedge\mu_2}^1 1_{A_1}\right]\\
\notag &\leq&\E\left[\E_{\rho_0\wedge\mu_1\wedge\mu_2}\left[v_{\mu_1\wedge\mu_2}^1\right] 1_{A_1}\right]\\
\notag &=&\E\left[v_{\mu_1\wedge\mu_2}^1 1_{A_1}\right]\\
\notag &=&\E\left[\left(v_{\mu_1}^1 1_{\{\mu_1\leq\mu_2\}}+v_{\mu_2}^1 1_{\{\mu_1>\mu_2\}}\right)1_{A_1}\right]\\
\notag &\leq&\E\left[\left(W_{\mu_1}^1 1_{\{\mu_1\leq\mu_2\}}+Y_{\mu_2}^1 1_{\{\mu_1>\mu_2\}}\right)1_{A_1}\right]+\eps.
\end{eqnarray}

\textbf{Case 2}: $A_2:=\{\mu_1\wedge\mu_2\leq\rho_0<\mu_1\wedge\mu_2+\delta\}\cap\{\mu_1\leq\mu_2\}$. By \leref{l3}
$$\E\left|W_{\mu_1+\delta}^1-W_{\mu_1}^1\right|\rightarrow 0,\quad\delta\searrow 0,$$
and
$$\PP\left\{([\mu_1/h]+1)h-\mu_1<2\delta\right\}\rightarrow 0,\quad\delta\searrow 0.$$
We choose $\delta$ small enough, such that $r(\delta)<\eps$, and
$$\E\left|W_{\mu_1+\delta}^1-W_{\mu_1}\right|<\eps$$
and
$$\PP\left\{([\mu_1/h]+1)h-\mu_1<2\delta\right\}<\eps/M,$$
where $M>0$ is a constant such that $|U^1|<M$ uniformly in $(s,t,\omega)$. Then
\begin{eqnarray}
\notag \E\left[U^1\left(\rho[\tau^*],\tau^*[\rho]\right)1_{A_2}\right]&=&\E\left[U^1\left(\rho_0,\tau_h^2(\rho_0)\right)1_{A_2}\right]\\
\notag &\leq&\E\left[U^1\left(\rho_0,\tau_h^2(\mu_1+\delta)\right)1_{A_2}\right]+2\eps\\
\notag &\leq&\E\left[U^1\left(\mu_1+\delta,\tau_h^2(\mu_1+\delta)\right)1_{A_2}\right]+3\eps\\
\notag &=&\E\left[W_{\mu_1+\delta}^1 1_{A_2}\right]+3\eps\\
\notag &\leq&\E\left[W_{\mu_1}^1 1_{A_2}\right]+4\eps\\
\notag &=&\E\left[\left(W_{\mu_1}^1 1_{\{\mu_1\leq\mu_2\}}+Y_{\mu_2}^1 1_{\{\mu_1>\mu_2\}}\right)1_{A_2}\right]+4\eps.
\end{eqnarray}

\textbf{Case 3}: $A_3:=\{\mu_1\wedge\mu_2\leq\rho_0<\mu_1\wedge\mu_2+\delta\}\cap\{\mu_1>\mu_2\}$. We choose $\delta$ small enough, such that $r(\delta)<\eps$ and
$$\E\left|Y_{\mu_2+\delta}^1-Y_{\mu_2}^1\right|<\eps.$$
Then we have that
\begin{eqnarray}
\notag \E\left[U^1\left(\rho[\tau^*],\tau^*[\rho]\right)1_{A_3}\right]&=&\E\left[\left(U^1\left(\mu_2,\mu_2\right)1_{\{\rho_0=\mu_2\}}+U^1\left(\rho_1(\mu_2),\mu_2\right)1_{\{\rho_0>\mu_2\}}\right)1_{A_3}\right]\\
\notag &\leq&\E\left[\left(Y_{\mu_2}^1 1_{\{\rho_0=\mu_2\}}+U^1\left(\rho_1(\mu_2)+\delta,\mu_2+\delta\right)1_{\{\rho_0>\mu_2\}}\right)1_{A_3}\right]+2\eps\\
\notag &\leq&\E\left[\left(Y_{\mu_2}^1 1_{\{\rho_0=\mu_2\}}+Y_{\mu_2+\delta}^1 1_{\{\rho_0>\mu_2\}}\right)1_{A_3}\right]+2\eps\\
\notag &\leq&\E\left[Y_{\mu_2}^11_{A_3}\right]+3\eps\\
\notag &=&\E\left[\left(W_{\mu_1}^1 1_{\{\mu_1\leq\mu_2\}}+Y_{\mu_2}^1 1_{\{\mu_1>\mu_2\}}\right)1_{A_3}\right]+3\eps.
\end{eqnarray}

\textbf{Case 4}: $A_4:=\{\rho_0\geq\mu_1\wedge\mu_2+\delta\}\cap\{\mu_1\leq\mu_2\}$. We choose $\delta$ small enough, such that
$$\E\left|v_{\mu_1+\delta}^1-v_{\mu_1}^1\right|<\eps.$$
By \reref{r1}, we have that
\begin{eqnarray}
\notag \E\left[U^1\left(\rho[\tau^*],\tau^*[\rho]\right)1_{A_4}\right]&=&\E\left[\E_{\mu_1+\delta}\left[U^1\left(\rho[\tau^*],\tau^*[\rho]\right)\right]1_{A_4}\right]\\
\notag &\leq&\E\left[v_{\mu_1+\delta}^1 1_{A_4}\right]+4\eps\\
\notag &\leq&\E\left[v_{\mu_1}^1 1_{A_4}\right]+5\eps\\
\notag &\leq&\E\left[W_{\mu_1}^1 1_{A_4}\right]+6\eps\\
\notag &=&\E\left[\left(W_{\mu_1}^1 1_{\{\mu_1\leq\mu_2\}}+Y_{\mu_2}^1 1_{\{\mu_1>\mu_2\}}\right)1_{A_4}\right]+6\eps.
\end{eqnarray}

\textbf{Case 5}: $A_5:=\{\rho_0\geq\mu_1\wedge\mu_2+\delta\}\cap\{\mu_1>\mu_2\}$. We choose $\delta$ small enough, such that $r(\delta)<\eps$, and
$$\E\left|Y_{\mu_2+\delta}^1-Y_{\mu_2}^1\right|<\eps.$$
Then
\begin{eqnarray}
\notag \E\left[U^1\left(\rho[\tau^*],\tau^*[\rho]\right)1_{A_5}\right]&=&\E\left[U^1\left(\rho_1(\mu_2),\mu_2\right)1_{A_5}\right]\\
\notag &\leq&\E\left[U^1\left(\rho_1(\mu_2)+\delta,\mu_2+\delta\right)1_{A_5}\right]+2\eps\\
\notag &\leq&\E\left[Y_{\mu_2+\delta}^1 1_{A_5}\right]+2\eps\\
\notag &\leq&\E\left[Y_{\mu_2}^1 1_{A_5}\right]+3\eps\\
\notag &=&\E\left[\left(W_{\mu_1}^1 1_{\{\mu_1\leq\mu_2\}}+Y_{\mu_2}^1 1_{\{\mu_1>\mu_2\}}\right)1_{A_5}\right]+3\eps.
\end{eqnarray}

From Cases 1-5, we have that
$$u^1(\rho,\tau^*)\leq u^1(\rho^*,\tau^*)+18\eps,$$
for any $\rho\in\kT$ by choosing $\delta$ small enough. Similar we can show that
$$u^2(\rho^*,\tau)\leq u^2(\rho^*,\tau^*)+18\eps,$$
for any $\tau\in\kT$ by choosing $\delta$ small enough.
\end{proof}

\bibliographystyle{siam}
\bibliography{ref}

\end{document}